\documentclass[11pt, reqno]{amsart}
\usepackage{ amsmath, amsthm, amscd, amsfonts, amssymb, graphicx, color, enumerate}
\usepackage[bookmarksnumbered, colorlinks, plainpages]{hyperref}

\definecolor{TITLE}{rgb}{0.0,0.0,1.0}
\definecolor{AUTHOR1}{rgb}{0.00,0.59,0.00}
\definecolor{AUTHOR2}{rgb}{0.50,0.00,1.00}
\definecolor{SECTION}{rgb}{0.50,0.00,1.00}
\definecolor{FOOTTITLE}{rgb}{0.00,0.50,0.75}
\definecolor{THM}{rgb}{0.7,0.3,0.3}
\definecolor{SEC}{rgb}{0.6,0.1,.5}

\makeatletter
\@namedef{subjclassname@2010}{%
\textup{2010} Mathematics Subject Classification}
\makeatother
\newtheorem{theorem}{{\color{THM} Theorem}}[section]

\newtheorem{lemma}[theorem]{{\color{THM}Lemma}}
\newtheorem{proposition}[theorem]{{\color{THM}Proposition}}
\newtheorem{corollary}[theorem]{{\color{THM}Corollary}}
\theoremstyle{definition}
\newtheorem{definition}[theorem]{{\color{THM}Definition\ }}

\newtheorem{remark}[theorem]{{\color{THM}Remark}}
\numberwithin{equation}{section}

\numberwithin{equation}{section}

\def\speaker{$ ^{*} $\protect\footnotetext{$ ^{*} $\lowercase{Corresponding Author.}}}
\begin{document}

\title {Banach Algebra of Complex Bounded Radon Measures on Homogeneous Space}

\author{ T. Derikvand\speaker$   $, R. A. Kamyabi-Gol and M. Janfada}

\address{International Campus, Faculty of Mathematic Sciences, Ferdowsi University of Mashhad, Mashhad, Iran}
\email{derikvand@miau.ac.ir}

\address{ Department of Pure Mathematics and  Centre of Excellence
in Analysis on Algebraic Structures (CEAAS), Ferdowsi University of
Mashhad, P.O. Box 1159, Mashhad 91775, Iran}
\email{kamyabi@um.ac.ir}

\address{ Department of Pure Mathematics, Ferdowsi University of
Mashhad, P.O. Box 1159, Mashhad 91775, Iran}
\email{Janfada@um.ac.ir}

\begin{abstract}
Let $ H $ be a compact subgroup of a locally compact group $G$. In this paper we define a convolution on $ M(G/H) $, the space of all complex bounded Radon measures on the homogeneous space G/H. Then we prove that the measure space $ M(G/H, *) $ is a non-unital Banach algebra that possesses an approximate identity. Finally, it is shown that the Banach algebra $ M(G/H, *) $ is not involutive and also $ L^1(G/H, *) $ is a two-sided ideal of it. \end{abstract}


\maketitle
                                                                                                                                                                                                                                                                                                                                                                                                                                                                                                                                                                                                                                                                                                                                                                                                                                                                                                                                                                                                                                                                                                                                                                                                                                                                                                                                                                                                                                                                                                                                                                                                                                                                                                                                                                                                                                                                                                                                                                                                                                                                                                                                                                                                                                                                                                                                                                                                                                                                                                                                                                                                                                                                                                                                                                                                                                                                                                                                                                                                                                                                                                                                                                                                                                                                                                                                                                                                                                                                                                                                                                                                                                                                                                                                                                                                                                                                                                                                                                                                                                                                                                                                                                                                                                                                                                                                                                                                                                                                                                                                                                                                                                                                                                                                                                                                                                                                                                                                                                                                                                                                                                                                                                                                                                                                                                                                                                                                                                                                                                                                                                                                                                                                                                                                                                                                                                                                                                                                                                                                                                                                                                                                                                                                                                                                                                                                                                                                                                                                                                                                                                                                                                                                                                        \section{INTRODUCTION AND PRELIMINARIES}
Let $G$ be a locally compact group. The convolution on $M(G)$, the space of all complex bounded Radon measures on $G$, is defined by
\begin{align}
\mu_{1} \ast \mu_{2}(f)=\int_{G}{\int_{G}{f(xy) d\mu_{1}(x) d\mu_{2}(y)}}\qquad (f\in C_{c}(G)),\label{eq11}
\end{align}
For any two complex bounded Radon measures $ \mu_{1}, \mu_{2} \in M(G) $. It is well-known that $ (M(G), \ast) $ is a unital Banach algebra, it is called the measure algebra. Now it is worthwhile to investigate how this can be done within $ M(G/H) $, where $H$ is a compact subgroup of locally compact group $G$. 
We should clear that $H$ is not normal subgroup necessarily, so $ G/H $ does not possess a group structure but it will be a locally compact Hausdorff  space.\\
Throughout  this paper $H$ is a compact subgroup of locally compact group $G$ and the homogeneous space $ G/H $ is a space on which $G$ acts transitively by left. The theory of homogeneous spaces has many applications in physics and engineering. For example, if the Euclidian group $E(2)$ acts transitively on $\mathbb{R}^2$, then the isotropy subgroup of origin is the orthogonal group $O(2)$. In that sequel, the homogeneous space $ E(2)/O(2) $ provides definition of X-ray transform that is used  in many areas such as radio astronomy, positron emission tomography, crystallography, etc (See e. g. \cite{1}, Ch. 1).\\
The modular function $ \triangle_{G} $ is a continuous homomorphism from $G$ into the multiplicative group $\mathbb{R}^{+} $. Furthermore, for all $ x\in G $
\[\int_{G} f(y)dy=\triangle_{G}(x)\int_{G} f(yx)dy\]
where $ f\in C_{c}(G) $, the space of continuous functions on $ G $ with compact support. A locally compact group $ G $ is called unimodular if $ \triangle_{G}(x)=1 $, for all $ x \in G $. A compact group $ G $ is always unimodular.
Assume that $H$ is a closed subgroup of the locally compact group $ G $, it is known that $C_c(G/H)$ consists of all $P_{H}f$ functions, where $f\in C_c(G)$ and
\[P_{H}f(xH)=\int_{H}f(xh)dh\quad(x\in G).\]
Moreover, $P_{H}:C_c(G)\rightarrow C_c(G/H)$ is a bounded linear operator which is not injective (See e. g. \cite{5}, Ch. 2, Section. 6). Suppose that $ \mu $ is a Radon measure on $ G/H $. For all $ x\in G $ we
define the translation of $ \mu $ by $ x $, by $ \mu_{x}(E)=\mu (xE) $, where $ E $ is a Borel subset of $ G/H $.
Then $ \mu $ is said to be $ G-$invariant if $ \mu_{x}=\mu $, for all $ x\in G $, if $ H $ is compact, $ G/H $ admits a $ G- $invariant Radon measure (See e. g. \cite{5}, Corollary 2. 51).\\
$ \mu $ is said to be strongly quasi-invariant, if there is a continuous function $ \lambda:G\times G/H\rightarrow (0, +\infty) $ which satisfies
\[d\mu_{x}(yH)=\lambda(x, yH)d\mu(yH).\]
If the function $ \lambda(x, .) $ is  reduced to a constant for each $ x \in G $, then $ \mu $ is called relatively invariant under $ G $. We consider a rho-function for the pair $ (G, H) $ as a continuous function $ \rho:G\rightarrow(0, +\infty) $ for which $\rho(xh)=\triangle_{H}(h)\triangle_{G}(h)^{-1}\rho(x)$, for each $x\in G$ and $h\in H$. It is well known that $ (G, H) $ admits a rho-function and for every rho-function $ \rho $ there is a strongly quasi-invariant measure $ \mu $ on $ G/H $ such that
\[\int_{G} f(x)dx=\int_{G/H}P_{H}f(xH)d\mu(xH) \qquad (f\in C_{c}(G)),\]
where, in this case, $ P_{H}f(xH)=\int_{H}\frac{f(xh)}{\rho(xh)} dh $ and the equation above is called quotient integral formula . The measure $\mu$ also satisfies
\[\frac{d\mu_{x}}{d\mu}(yH)=\frac{\rho(xy)}{\rho(y)}\qquad (x,y \in G).\]
If $ \mu $ is a strongly quasi invariant measure on $G/H$ which is associated with the rho$-$function $\rho$ for the pair $ (G, H) $, then the mapping $T_{H}: L^{1}(G)\rightarrow L^{1}(G/H)$ is defined almost everywhere by
\[T_{H}f(xH)=\int_{H}\frac{f(xh)}{\rho(xh)}dh \qquad (f\in L^{1}(G))\],
is a surjective bounded linear operator with $\parallel T_{H} \parallel \leq1$ (see \cite{8}, Subsection 3.4) and also satisfies the generalized Mackey-Bruhat formula,
\begin{align}
\int_{G}f(x)dx=\int_{G/H}T_{H}f(xH)d\mu(xH)\qquad (f\in L^{1}(G)),\label{eq21}
\end{align}
which is also known as the quotient integral formula.\\
Two useful operators left translation and right translation, denoted by $ \emph{L} $ and $ \emph{R} $ respectively, plays crucial role in what follows. The left translation of $ \varphi $ by $ x \in G $ is defined by $ \emph{L}_{x} (\varphi)(yH)=\varphi(x^{-1}yH) $, where $ \varphi \in C_{c}(G/H) $. In a similar way, the left translation operator is defined for the integrable function on a homogeneous space $ G/H $ as follows:
\[\emph{L}_{x} (\varphi)(yH)=\varphi(x^{-1}yH) \qquad (\mu-almost all yH \in G/H)\],
where $ \varphi \in L^{p}(G/H) $, $ 1 \leq p \leq \infty $. The mapping $ x\mapsto \emph{L}_{x} (\varphi) $ is continuous and also $ \Vert \emph{L}_{x} (\varphi) \Vert_{p}=\left( \frac{\rho(x)}{\rho(e)}\right)^{1/p} \Vert \varphi \Vert_{p} $. The right translation is defined in the same manner. For more details see \cite{7}.
Now, let $H$ be a compact subgroup and consider 
\begin{center}
$ C_c(G  :  H):=\lbrace f \in C_c(G) : R_{h}f=f $,   $ \forall h \in H\rbrace$,\\
\end{center}
where $ R_{h} $ denotes the right translation through $ h $.  
Let $ \mu $ be a $ G-$invariant Radon measure on $ G/H $. One can easily prove that
\begin{center}
$ C_c(G  :  H)=\lbrace \varphi_{\pi_{H}}:=\varphi \circ \pi_{H} : \varphi \in C_c(G/H) \rbrace$,\\
\end{center}
is a left ideal of the algebra $ C_c(G) $ and $ P_{H} $ is an algebraic isometric isomorphism between $C_c(G  :  H)$ and $C_c(G/H)$. Furthermore, $  P_{H}(\varphi_{\pi_{H}})=\varphi $, for all $\varphi \in C_c(G/H)$. 
These results can be extended, by approximation, to $T_{H}:L^{1}(G  :  H)\rightarrow L^{1}(G/H)$, where  
\begin{center}
$ L^{1}(G  :  H):=\lbrace f \in L^{1}(G) : R_{h}f=f $,   $ \forall h \in H\rbrace$,
\end{center}
(See e. g. \cite{8}, P. 98) and also (See e. g. \cite{3}, \cite{7}).
Therein $T_{H}$ is an algebraic isometrically isomorphism. By using this isomorphism one can define a well-defined convolution on $ L^{1}(G/H) $. Let $ \lambda $ ba a strongly quasi-invariant measure on $ G/H $ that arises from the rho-function $ \rho $ then:
\begin{align*}
\varphi \ast \psi(xH)
&=T_{H}(\varphi_{\pi_{H}}\ast \psi_{\pi_{H}})(xH)=\int_{G/H}{\int_{H}{\varphi(yH)\psi(hy^{-1}xH)\frac{\rho(hy^{-1}x)}{\rho(x)}dhd\lambda(yH)}}.
\end{align*}
Now, let $ M(G) $ be the space of all complex bounded Radon measures on locally compact group $G$ and $H$ be a compact subgroup of $G$ and also assume that $ \mu \in M(G) $. One can define $ \sigma_{\mu} \in M(G/H) $ by
\begin{align}
\int_{G/H}\varphi(xH)d\sigma_{\mu}(xH)=\int_{G} \varphi_{{\pi}_{H}}(x)d \mu(x) \qquad (\varphi \in C_{c}(G/H)).\label{eq13}
\end{align}
In other words $ \sigma_{\mu} ( \varphi )= \mu ( \varphi_{\pi_{H}}) $, for all $ \varphi $ in $ C_{c}(G/H) $. Since $ \Vert \sigma_{\mu} \Vert  \leqslant \Vert \mu \Vert $,  the linear map $ \mu \mapsto \sigma_{\mu} $ is continuous and it can be shown that this map is surjective (See e. g. \cite{8}, P. 233).
\section{THE MAIN RESULTS}
Let us denote the space of all complex bounded Radon measures on locally compact Hausdorff space $G/H$ by  $ M(G/H) $. In this section we establish some results to define a well-defined convolution on $ M(G/H) $ which makes it to a Banach algebra, then we introduce an approximate identity for it; After that the relationship between two Banach algebras $ M(G/H) $ and $ L^{1}(G/H) $ is described, the last result asserts that $ L^{1}(G/H) $ can be regarded as a Banach subalgebra of $ M(G/H) $.\\
From now on, we consider $ H $ as a compact subgroup of locally compact group $ G $.\\
We first introduce two crucial subalgebras of the measure algebra $ M(G) $. Consider the following notation for the space of $ C_c(G : H)- $invariant measures in $ M(G) $.									\begin{center}
$ M_{H}(G):=\lbrace \mu \in M(G) : \mu(f*\varphi_{\pi_{H}})=\mu(f) $,   $\forall f \in C_c(G),$   $ \forall \varphi_{\pi_{H}} \in C_c(G : H) \rbrace$.\\
\end{center}
\begin{proposition}\label{pr1}
Let $H$ be a compact subgroup of a locally compact group $ G $. Then $ M_{H}(G) $ is a closed left ideal of $ M(G) $.
\end{proposition}
 \begin{proof}
Let $ \mu_{1}, \mu_{2} \in M_{H}(G) $ and also $ f \in C_c(G) $. Then, for all $\varphi_{\pi_{H}}$ we have
\begin{align*}
\mu_{1} \ast \mu_{2}(f \ast \varphi_{\pi_{H}})
&=\int_{G}{f \ast \varphi_{\pi_{H}} (g) d(\mu_{1} \ast \mu_{2})(g)}\\
&=\int_{G}{\int_{G}{f \ast \varphi_{\pi_{H}}(xy) d\mu_{1}(x) d\mu_{2}(y)}}\\
&=\int_{G}{\int_{G}{f \ast \varphi_{\pi_{H}}(xy) d\mu_{2}(x) d\mu_{1}(y)}}\\
&=\int_{G}{\int_{G}{L_{x^{-1}}(f \ast \varphi_{\pi_{H}})(y) d\mu_{2}(x) d\mu_{1}(y)}}\\
&=\int_{G}{\int_{G}{((L_{x^{-1}}f) \ast \varphi_{\pi_{H}})(y) d\mu_{1}(x) d\mu_{2}(y)}}\\
&=\int_{G}{\int_{G}{(L_{x^{-1}}f)(y) d\mu_{1}(x) d\mu_{2}(y)}}\\
&=\int_{G}{\int_{G}{f(xy) d\mu_{1}(x) d\mu_{2}(y)}}\\
&=\mu_{1} \ast \mu_{2}(f).
\end{align*}
Therefore $ \mu_{1} \ast \mu_{2} \in M_{H}(G) $. 
In a similar way, a direct calculation shows that    $ M_{H}(G) $ is a left ideal of $ M(G) $. Furthermore, let $ \lbrace \mu_{\alpha} \rbrace_{\alpha\in\Lambda} $ be a net of elements of $M_{H}(G)$ approaches to $ \mu $ in $ M_{H}(G) $, then for all $ f \in C_c(G) $ and $\varphi_{\pi_{H}} \in C_c(G : H) $ we have
\begin{center}
$ \mu(f \ast \varphi_{\pi_{H}})=$lim$\mu_{\alpha}(f \ast \varphi_{\pi_{H}})=$lim$\mu_{\alpha}(f)=\mu(f)$.
\end{center}
 Hence the proof is complete.
\end{proof}
Another important closed left ideal of $ M(G) $ will be introduced in what follows. Let
\begin{center}
$ M(G : H)=\lbrace \mu \in M(G) : \mu(R_{h}f)=\mu(f) ; \forall f \in C_c(G), h \in H \rbrace $,
\end{center}
where $ R_{h} $ denotes the right translation through $ h $. 

\begin{proposition}\label{pr22}
Let $H$ be a compact subgroup of a locally compact group $ G $ and $ \mu $ be a left Haar measure on $ G $. Then 
\begin{center}
$ M(G : H)=\lbrace \mu_{f} : f \in L^{1}(G : H) \rbrace $,
\end{center}
where $ d\mu_{f}(x)=f(x)d\mu(x) $.
\end{proposition}
\begin{proof}
For any $ f \in L^{1}(G : H) $, it is clear that $ \mu_{f}(x)=f(x)d\mu(x) \in M(G) $ and also for all $ g \in C_c(G) $ and $ h \in H $ we have
\begin{align*}
\mu_{f}(R_{h}g)
&=\int_{G}{R_{h}g(x)d\mu_{f}(x)}\\
&=\int_{G}{g(xh)f(x)d\mu(x)}\\
&=\int_{G}{g(x)f(xh^{-1})d\mu(xh^{-1})}\\
&=\int_{G}{g(x)d\mu_{f}(x)}\\
&=\mu_{f}(g).
\end{align*}
Note that since $ H $ is compact, $ \Delta_{G}\vert_{H}=1 $.\\
Now let $ \mu_{f} \in M(G : H) $ for some $ f \in L^{1}(G) $, so $ \mu_{f}(R_{h}g)=\mu_{f}(g) $ for all $ g \in C_c(G) $. In other words, $ \int_{G}{R_{h}g(x)d\mu_{f}(x)}=\int_{G}{g(x)d\mu_{f}(x)} $. Since
\begin{align*}
\int_{G}{g(xh)f(x)d\mu(x)}=\int_{G}{g(x)f(x)d\mu(x)},
\end{align*}
we have 
\begin{align*}
\int_{G}{g(x)f(xh^{-1})d\mu(xh^{-1})}=\int_{G}{g(x)f(xh^{-1})d\mu(x)}=\int_{G}{g(x)f(x)d\mu(x)},
\end{align*}
for all $ g \in C_c(G) $. Therefore, $ \int_{G}{g(x) \left( f(xh)-f(x) \right)d\mu(x)}=0$ for all $ g \in C_c(G) $ and $ h \in H$. Then by Urysohn's Lemma to take suitable $ g \in C_c(G) $ we get $ f(xh)=f(x) $, for all $ x \in G $ and  $ h \in H $. Thus $ f\in L^{1}(G : H) $.
\end{proof}
\begin{proposition}\label{pr2}
Let $H$ be a compact subgroup of a locally compact group $ G $. Then $ M(G : H) $ is a closed left ideal of $ M(G) $. Moreover,
\begin{center}
$ M(G : H)=\lbrace \sigma_{P_{H}}:=\sigma \circ P_{H} : \sigma \in M(G/H) \rbrace $.
\end{center}
\end{proposition}
\begin{proof}
Let $ \mu_{1}, \mu_{2} \in M(G : H)$, then for all $ f \in C_c(G) $ and $ h \in H $ we have
\begin{align*}
\mu_{1} \ast \mu_{2}(R_{h}f)
&=\int_{G}{R_{h}f (g) d(\mu_{1} \ast \mu_{2})(g)}\\
&=\int_{G}{\int_{G}{R_{h}f(xy) d\mu_{1}(x) d\mu_{2}(y)}}\\
&=\int_{G}{\mu_{1}(R_{h}(R_{y}f))(y) d\mu_{2}(y)}\\
&=\int_{G}{\mu_{1}(R_{y}f)(y) d\mu_{2}(y)}\\
&=\int_{G}{\int_{G}{f(xy) d\mu_{1}(x) d\mu_{2}(y)}}\\
&=\mu_{1} \ast \mu_{2}(f).
\end{align*}
Therefore $ \mu_{1} \ast \mu_{2} \in M(G : H) $. 
A similar calculation shows that    $ M(G : H) $ is a left ideal of $ M(G) $. Furthermore, let $ \mu $ be limit of  the net $ \lbrace \mu_{\alpha} \rbrace_{\alpha\in\Lambda} $ in $M(G : H)$, then $ \mu(R_{h}f)=$lim$\mu_{\alpha}(R_{h}f)$  for all $ f \in C_c(G) $ and $ h \in H $. But $\mu_{\alpha}(R_{h}f)=\mu_{\alpha}(f)$ and this implies that $ \mu(R_{h}f)=\mu(f) $. It remains to prove the equality in this Proposition. Let $\sigma \in M(G/H)$ then 
$ \sigma \circ P_{H} $ is a bounded linear functional on $ C_c(G) $, since
\begin{align*}
\vert \sigma \circ P_{H}(f) \vert
&=\vert \sigma(P_{H}(f)) \vert\\
&=\vert \int_{G/H}{P_{H}f(xH)d \sigma(xH)} \vert\\
&=\int_{G/H}{\vert \int_{H}{f(xh) dh\vert d\vert \sigma \vert (xH)}} \\
&\leqslant \Vert f \Vert_{\infty}\Vert \sigma \Vert.
\end{align*}
Thus $\Vert \sigma \circ P_{H} \Vert \leq \Vert \sigma \Vert < \infty$  ), so that the mapping $ \sigma \circ P_{H} $ is a bounded linear functional on $ C_{c}(G) $. Furthermore, for all $ f \in C_c(G) $ and $ h \in H $ we have
\begin{align*}
\sigma \circ P_{H}(R_{h}f)
&=\int_{G/H}{\int_{H}{R_{h}f(x \eta) d \eta d \sigma(xH)}}\\
&=\int_{G/H}{\int_{H}{f(x \eta h) d \eta d\sigma(xH)}}\\
&=\int_{G/H}{\int_{H}{f(x \eta) d \eta d\sigma(xH)}}\\
&=\sigma \circ P_{H}(f).
\end{align*}
Thus $ \sigma_{P_{H}}=\sigma \circ P_{H} \in M(G : H) $ for all $ \sigma \in M(G/H) $. To show the reverse inclusion let $\mu$ be in $M(G)$ such that $ \mu(R_{h}f)=\mu(f)$ for all $ f $ in $ C_c(G) $ and $ h $ in $H$. Then by \eqref{eq13} there exists $ \sigma \in M(G/H) $ such that for all $ f $ in $ C_c(G)$ we have
\begin{align*}
\mu(f)
&=\int_{G}{f(x) d\mu(x)}\\
&=\int_{G}{R_{h}f(x) d\mu(x)}\\
&=\int_{G/H}{\int_{H}{f(x \eta h ) d \eta d\sigma(xH)}}\\
&=\int_{G/H}{\int_{H}{f(x \eta) d \eta d\sigma(xH)}}\\
&=\sigma \circ P_{H}(f).
\end{align*}
So $ \mu=\sigma \circ P_{H}$
and the proof is complete.
\end{proof}
Now, consider the map $R_{H}:M(G) \rightarrow M(G/H)$ given by
\begin{align}
R_{H}  \mu(\varphi):=\mu(\varphi_{\pi_{H}})=\int_{G}\varphi_{\pi_{H}}(x)d\mu(x)\quad(\varphi \in C_{c}(G/H)). \label{eq21}
\end{align}
Let $ \varphi=\psi \in C_{c}(G/H)) $ then $ \varphi_{\pi_{H}}=\varphi \circ \pi_{H}=\psi \circ \pi_{H}=\psi_{\pi_{H}} $. Hence $ \mu(\varphi_{\pi_{H}})=\mu(\psi_{\pi_{H}}) $ and this implies that $R_{H}  \mu(\varphi)=R_{H}  \mu(\psi)$. From the definition we can easily deduce that $ R_{H} \mu $ is a positive linear functional on $ C_{c}(G/H) $. So by the Riesz representation theorem there exists a unique Radon measure $ \sigma \in M(G/H) $ such that
\begin{align}
R_{H} \mu(\varphi)=\int_{G/H}\varphi(xH)d\sigma(xH)=\sigma(\varphi). \label{eq22} 
\end{align}
Then $ R_{H} \mu = \sigma \in M(G/H) $. Also based on the definition \eqref{eq21} it is clear that $  R_{H} (\mu_{1}) =R_{H}( \mu_{2}) $ if $ \mu_{1} = \mu_{2}$. So $ R_{H}$ is a well-defined map. To show that the mapping $ R_{H} $ is linear, consider an arbitrary scalar $\alpha$ and the elements $ \mu_{1} $ and $ \mu_{2} $ in $ M(G) $. Then for any $  \varphi $ in $ C_{c}(G/H) $ we have
\begin{align*}
R_{H}  (\mu_{1}+\mu_{2})(\varphi)
&=(\mu_{1}+\mu_{2})(\varphi_{\pi_{H}})\\
&= \mu_{1}(\varphi_{\pi_{H}})+\mu_{2}(\varphi_{\pi_{H}})\\ &=R_{H} \mu_{1}(\varphi) + R_{H}\mu_{2}(\varphi)\\
&=(R_{H} \mu_{1}+ R_{H}\mu_{2})(\varphi),
\end{align*}
thus $ R_{H} $ is linear. We shall show that $ R_{H} $ is a bounded operator. To do this, if we consider any $ \varphi $ in $ C_{c}(G/H) $ then we have\\
\begin{align*}
\vert R_{H}  \mu(\varphi) \vert
=\vert \int_{G} \varphi_{{\pi}_{H}}(x)d \mu(x) \vert
\leqslant \int_{G} \vert \varphi(xH) \vert d \vert \mu \vert (x)
\leqslant \int_{G} \Vert \varphi \Vert_{\infty} d \vert \mu \vert (x)
\leqslant \Vert \mu \Vert \Vert \varphi \Vert_{\infty}  .
\end{align*}
So $ \Vert R_{H} \mu \Vert \leqslant \Vert \mu \Vert < \infty$ and this implies $ \Vert R_{H} \Vert \leqslant 1$.											For surjectivity, let $ \sigma \in M(G/H) $ and define $ \mu $ on $ C_{c}(G) $ by
\begin{align}
\mu(f):=\sigma(P_{H}f) \quad(f \in C_{c}(G)). \label{eq23}
\end{align}
Suppose  $ \varphi \in C_{c}(G/H) $, using the Proposition \ref{pr2},for $ \mu \in M(G : H) $. Then by the definition of $ R_{H} $, we have  
\begin{align*}
R_{H} \mu(\varphi)
&=\mu(\varphi_{\pi_{H}} )\\
&=\sigma (P_{H}(\varphi_{\pi_{H}}))\\
&=\int_{G/H}P_{H}(\varphi_{\pi_{H}})(xH)d\sigma(xH)\\
&=\int_{G/H}\varphi(xH)d\sigma(xH)\\
&=\sigma(\varphi),
\end{align*}
this proves surjectivity.
\begin{remark}\label{Re1}
The operator $R_{H}$ is an extension of the mapping $T_{H}:L^{1}(G)\rightarrow L^{1}(G/H)$ given by $T_{H}f(xH)=\int_{H}f(xh)dh$, for all $x\in G.$
\end{remark}

The next two  Propositions play a central role for making  $M(G/H)$ into a Banach algebra. 
\begin{proposition}\label{pr3}
Let $ H $ be a compact subgroup of a locally compact group $ G $. Then $ R_{H}\mid_{M(G : H)}$, the restriction of $R_{H}$ to $ M(G : H)$, is a bijective mapping and also it is an isometry.                                                                                                                                          
\end{proposition}
 \begin{proof}
Since $ R_{H} $ is surjective, it is enough to show that it is injective.
Let $ \mu \in M(G : H) $ and $ R_{H}(\mu)=0 $. Then there exists $\sigma \in M(G/H) $ such that $ \mu=\sigma_{P_{H}}=\sigma \circ P_{H} $ and $ R_{H}(\mu)=0 $ implies that for all $ \varphi \in C_c(G/H) $, $\sigma(\varphi)=\sigma_{P_{H}}(\varphi_{\pi_{H}})=0 $. So that $\mu=0$ and therefore $ R_{H} $ is injective.\\
Let $\sigma_{P_{H}}$ be in $M(G : H)$, then for all $ \varphi $ in $ C_c(G/H) $, on  one hand\\
\begin{align*}
\vert R_{H}(\sigma_{P_{H}})\varphi \vert
=\vert \sigma_{P_{H}}(\varphi_{\pi_{H}}) \vert
\leqslant \Vert \sigma_{P_{H}} \Vert \Vert \varphi_{\pi_{H}} \Vert_{\infty},
\end{align*}
so $\Vert R_{H}\sigma_{P_{H}} \Vert \leqslant \Vert \sigma_{P_{H}} \Vert $
and on the other hand, 
\begin{align*}
\vert \sigma_{P_{H}}(\varphi_{\pi_{H}}) \vert
&=\vert R_{H}\sigma_{P_{H}}(\varphi) \vert\\
&\leqslant \Vert R_{H}\sigma_{P_{H}} \Vert \Vert \varphi \Vert\\
&=\Vert R_{H}\sigma_{P_{H}} \Vert \Vert P_{H}(\varphi_{\pi_{H}}) \Vert\\
&\leqslant \Vert R_{H} \sigma_{P_{H}} \Vert \Vert P_{H}\Vert\Vert \varphi_{\pi_{H}} \Vert\\
&\leqslant \Vert R_{H} \sigma_{P_{H}} \Vert \Vert \varphi_{\pi_{H}} \Vert.
\end{align*}
so $ \Vert \sigma_{P_{H}} \Vert \leqslant \Vert R_{H}\sigma_{P_{H}} \Vert $. 
Hence the proof is complete.
\end{proof}
The remarkable equality $ R_{H}(\delta_{x})=\delta_{xH} $ is obtained by using the following equalities:
\begin{align*}
R_{H}(\delta_{x})(\varphi)
&=\delta_{x}(\varphi_{\pi_{H}})\\
&=\int_{G}{\varphi_{\pi_{H}}(y) d\delta_{x}(y)}\\
&=\varphi_{\pi_{H}}(x)\\
&=\varphi(xH)\\
&=\int_{G/H}{\varphi(yH) d\delta_{xH}(yH)}\\
&=\delta_{xH}(\varphi),
\end{align*}

for all $ x\in G $. Note that for all $ \varphi $ in $ C_c(G/H) $ and $ x\in G $ we get
\begin{align*}
\delta_{xH}(\varphi)
=\int_{G/H}{\varphi(yH) d\delta_{xH}(yH)}
=\varphi(x\eta H)
=\varphi(xH).
\end{align*}
where $ \eta \in H $.\\
Now, we are able to define a convolution on $ M(G/H) $.
\begin{definition}\label{def6}
Let $H$ be a compact subgroup of a locally compact group $G$. The mapping $ \ast:M(G/H) \times M(G/H) \rightarrow M(G/H)$ given by
\begin{align}
\sigma_{1} \ast \sigma_{2} (\varphi):=R_{H}(\sigma_{1_{P_{H}}}\ast \sigma_{2_{P_{H}}})(\varphi) \quad(\varphi \in C_{c}(G/H)), \label{eq24}
\end{align}
is a well-defined convolution on $ M(G/H) $.
\end{definition}
To show that $ \ast $ is well defined, let $ \sigma_{1}, \sigma_{2}, \sigma^{'}_{1}, \sigma^{'}_{2} \in M(G/H)$ and $ ( \sigma_{1}, \sigma_{2})=(\sigma^{'}_{1}, \sigma^{'}_{2})$. Using surjectivity of $R_{H}$, there exists $ \sigma_{1_{P_{H}}}, \sigma_{2_{P_{H}}}, \sigma^{'}_{1_{P_{H}}}, \sigma^{'}_{2_{P_{H}}} \in M(G : H) $ such that
\begin{center}
$ R_{H}(\sigma_{1_{P_{H}}})=\sigma_{1}, R_{H}(\sigma_{2_{P_{H}}})=\sigma_{2}, R_{H}(\sigma^{'}_{1_{P_{H}}})=\sigma^{'}_{1}, R_{H}(\sigma^{'}_{2H})=\sigma^{'}_{2} $.
\end{center}
 Therefore the injectivity of $R_{H}$ implies that $( \sigma_{1_{P_{H}}}, \sigma_{2_{P_{H}}})=(\sigma^{'}_{1_{P_{H}}}, \sigma^{'}_{2_{P_{H}}})$. Thus
 \begin{center}
$ \sigma_{2_{P_{H}}} \ast \sigma_{2_{P_{H}}}=\sigma^{'}_{1_{P_{H}}} \ast \sigma^{'}_{2_{P_{H}}},$ 
 \end{center}
 since the convolution on $ M(G) $ is well-defined. Then $ R_{H}(\sigma_{1_{P_{H}}} \ast \sigma_{2_{P_{H}}})=R_{H}(\sigma^{'}_{1_{P_{H}}} \ast \sigma^{'}_{2_{P_{H}}})$. Finally, by \eqref{eq24}, $ \sigma_{1}\ast \sigma_{2}= \sigma^{'}_{1} \ast \sigma^{'}_{2}$. 
Consequently, convolution $ \ast $ is well-defined.\\
Using Proposition \ref{pr3} and definition \ref{def6} we deduce the following result. 
\begin{corollary}
The bijective mapping $ R_{H}\mid_{M(G : H)} $ in the Proposition \ref{pr3} is an algebraic isometric isomorphism.                                                                                                                                       
\end{corollary}
Now some remarks are in orders.
\begin{remark}\label{re7}
With the notations as above, we have:
\begin{enumerate}
\item[(i)] $(\sigma_{1}\ast \sigma_{2})_{P_{H}}=\sigma_{1_{P_{H}}}\ast \sigma_{2_{P_{H}}}$, because $R_{H}(\sigma_{1_{P_{H}}}\ast \sigma_{2_{P_{H}}})=R_{H}((\sigma_{1}\ast \sigma_{2})_{P_{H}})$ and $ R_{H} $ is one to one on $ M(G : H) $.
\item[(ii)] One can simplify \eqref{eq24} as follows:
\begin{align*}
\sigma_{1} \ast \sigma_{2} (\varphi)
&=R_{H}(\sigma_{1_{P_{H}}}\ast \sigma_{2_{P_{H}}})(\varphi) \\
&=\sigma_{1_{P_{H}}}\ast \sigma_{2_{P_{H}}}(\varphi_{\pi_{H}})\\
&=\int_{G}{\int_{G}{\varphi_{\pi_{H}}(xy) d\sigma_{1_{P_{H}}}(x) d\sigma_{2_{P_{H}}}(y)}}\\
&=\int_{G}{\int_{G}{\varphi(xyH) d\sigma_{1_{P_{H}}}(x) d\sigma_{2_{P_{H}}}(y)}},
\end{align*}
for all $ \varphi \in C_{c}(G/H).$
\item[(iii)] Iet $ \mu\in M(G) $ and $ \sigma \in M(G/H) $, if we define $ \mu \ast \sigma :=R_{H}(\mu\ast \sigma_{P_{H}}) $ then we have
\begin{align*}
\mu \ast \sigma (\varphi)
&=R_{H}(\mu\ast \sigma_{P_{H}})(\varphi)\\
&=\mu \ast \sigma_{P_{H}}(\varphi_{\pi_{H}})\\
&=\int_{G}{\int_{G}{\varphi_{\pi_{H}}(xy) d \mu(x) d\sigma_{P_{H}}(y)}}\\
&=\int_{G}{\int_{G}{\varphi_{\pi_{H}}(xy) d\sigma_{P_{H}}(y)d \mu(x)}}\\
&=\int_{G}{\int_{G}{(L_{x^{-1}}\varphi_{\pi_{H}})(y) d\sigma_{P_{H}}(y) d\mu(x)}}\\
&=\int_{G}{\int_{G/H}{P_{H}(L_{x^{-1}}\varphi_{\pi_{H}})(yH) d\sigma(yH) d\mu(x)}}\\
&=\int_{G}{\int_{G/H}{\int_{H}{L_{x^{-1}}\varphi_{\pi_{H}}(yh)dh d\sigma(yH) d\mu(x)}}}\\
&=\int_{G}{\int_{G/H}{\varphi_{\pi_{H}}(xy)d\sigma(yH) d\mu(x)}}\\
&=\int_{G/H}{\int_{G}{\varphi(xyH)d\mu(x) d\sigma(yH)}},
\end{align*}
for all $ \varphi \in C_{c}(G/H).$
\end{enumerate}
\end{remark}
By using part (iii) of the Remark \ref{re7} it is deduced that $ M(G/H) $ is a left $ M(G) $ module.
In the next main theorem, it is shown that $ M(G/H, *) $ is a  Banach algebra and has an approximate identity. \begin{theorem}\label{th5}
 $ (M(G/H), \ast) $ is a Banach algebra and also it possesses an approximate identity.
\end{theorem}
\begin{proof}
It is well known that $ M(G/H) $ endowed with the total variation norm is a Banach space (See e. g. \cite{8}, P. 233).
The fact that convolution on $M(G/H)$ is  associative follows by applying $ (ii) $ of Remark \ref{re7} twice and associativity of $M(G)$.
Let $ \sigma_{1}, \sigma_{2}$ be in $M(G/H)$. Then by using surjectivity of $R_{H}$, there exists $ \sigma_{1_{P_{H}}}$ and, $\sigma_{2_{P_{H}}}$ in $M_{H}(G)$ such that $ R_{H}(\sigma_{1_{P_{H}}})=\sigma_{1}$ and $R_{H}(\sigma_{2_{P_{H}}})=\sigma_{2}$. Now the definition \eqref{def6} and the fact that $M(G : H)$ is an normed algebra imply that:
\begin{align*}
\Vert \sigma_{1} \ast \sigma_{2} \Vert
=\Vert \sigma_{1_{P_{H}}} \ast \sigma_{2_{P_{H}}}\Vert
\leqslant \Vert \sigma_{1_{P_{H}}}\Vert \Vert \sigma_{2_{P_{H}}}\Vert
\leqslant \Vert R_{H}\sigma_{1_{P_{H}}} \Vert \Vert  R_{H}\sigma_{2_{P_{H}}}\Vert
=\Vert \sigma_{1}\Vert \Vert \sigma_{2}\Vert.
\end{align*}
Note that $ R_{H} $ is an isometry. Thus $ (M(G/H), \ast) $ is a normed Banach algebra.\\
To introduce an approximate identity, let $ \lbrace \varphi_{\alpha}\rbrace_{\alpha\in\Lambda} $ be an approximate identity for the Banach algebra $ L^{1}(G/H) $, see \cite{3}.  Put $ \sigma_{\alpha}:=R_{H}(\mu_{(\varphi_{\alpha})_{\pi_{H }}} $, for all $ \alpha \in \Lambda $ where $ \mu $ is the left Haar measure on $ G $. Then by surjectivity of $ R_{H} $, for any $ \sigma $ in $ M(G/H) $ there exists $ \sigma_{P_{H}}\in M(G : H) $ such that $ R_{H}(\sigma_{P_{H}})=\sigma $. Hence we have
\begin{align*}
\Vert \sigma_{\alpha} \ast \sigma -\sigma\Vert
=\Vert R_{H}(\mu_{(\varphi_{\alpha})_{\pi_{H }}}) \ast R_{H}(\sigma_{P_{H}}) - R_{H}(\sigma_{P_{H}})\Vert
=\Vert \mu_{(\varphi_{\alpha})_{\pi_{H }}} \ast \sigma_{P_{H}} - \sigma_{P_{H}} \Vert,
\end{align*}
but by Proposition \ref{pr22} there exists $ \psi_{\pi_{H}} \in C_{c}(G : H)  $ such that $ \sigma_{P_{H}}= \mu_{\psi_{\pi_{H}}} $ and also it can be seen by direct computation that $ \mu_{f}*\mu_{g}-\mu_{h}=\mu_{f \ast g - h}$, for all $ f, g $ and $ h $ in $ C_{c}(G), $   so
\begin{align*}
\Vert \sigma_{\alpha} \ast \sigma -\sigma\Vert
=\Vert \mu_{(\varphi_{\alpha})_{\pi_{H }}} \ast \mu_{\psi_{\pi_{H}}} - \mu_{\psi_{\pi_{H}}} \Vert
=\Vert \mu_{(\varphi_{\alpha})_{\pi_{H }} \ast \psi_{\pi_{H}}-\psi_{\pi_{H}}} \Vert.
\end{align*}
On the othere hand, the embedding of $ L^{1}(G) $ into $ M(G) $ is isometric, therefore
\begin{align*}
\Vert \sigma_{\alpha} \ast \sigma - \sigma\Vert
&=\Vert (\varphi_{\alpha})_{\pi_{H }} \ast \psi_{\pi_{H}}-\psi_{\pi_{H}} \Vert\\
&=\Vert P_{H}\left( (\varphi_{\alpha})_{\pi_{H }} \ast \psi_{\pi_{H}}-\psi_{\pi_{H}}\right)   \Vert\\
&=\Vert P_{H}\left( (\varphi_{\alpha}) _ {\pi_{H }} \ast \psi_{\pi_{H}}\right) -P_{H}\left(\psi_{\pi_{H}}\right)   \Vert\\
&=\Vert P_{H}\left( (\varphi_{\alpha}) _ {\pi_{H }}\right)  \ast P_{H} \left( \psi_{\pi_{H}}\right) -P_{H}\left(\psi_{\pi_{H}}\right)   \Vert\\
&=\Vert \varphi_{\alpha} \ast \psi - \psi   \Vert, 
\end{align*}
Since $ \lbrace \varphi_{\alpha}\rbrace_{\alpha\in\Lambda} $ is an approximate identity, $ \Vert \varphi_{\alpha} \ast \psi - \psi   \Vert $  tends to $ 0 $ as $ \alpha \rightarrow \infty $. Note that in the two last equalities,  $ P_{H} $ is an isometry from $ C_{c}(G : H) $ onto $ C_{c}(G/H) $.\\
This implies that $ \Vert \sigma_{\alpha} \ast \sigma - \sigma\Vert $  goes to $ 0 $ when $ \alpha \rightarrow \infty $.
\end{proof}
In the sequel, consider $ \delta_{e} $ as the unit element of the unital Banach algebra $ M(G) $. If we define the point mass measure $ \delta_{H}:=R_{H}(\delta_{e}) $, then for all $ \varphi $ in $ C_c(G/H) $, by the definition of $ R_{H} $ we have
\begin{align}
\delta_{H}(\varphi)=R_{H}(\delta_{e})(\varphi)=\delta_{e}(\varphi_{\pi_{H}})=\int_{G}{\varphi_{\pi_{H}}(x)d\delta_{e}(x)}=\varphi_{\pi_{H}}(e)=\varphi(H).\label{eq25}
\end{align}
Note that for all $ \varphi $ in $ C_c(G/H) $ we have
\begin{center}
$\delta_{H}(\varphi)=\int_{G/H}{\varphi(xH)d\delta_{H}(xH)}=\varphi(H).$
\end{center}
\begin{lemma}\label{le11}
Let $H$ be a compact subgroup of a locally compact group $G$. $ \delta_{H} $ is a right multiplicative identity in the algebra $ M(G/H) $.
\end{lemma}
\begin{proof}
Suppose that $ \varphi $ is in $ C_c(G/H) $ and $ \sigma \in M(G/H) $. Then we have
\begin{align*}
\sigma \ast \delta_{H}(\varphi)
&=R_{H}(\sigma_{P_{H}} \ast (\delta_{H})_{P_{H}})(\varphi)\\
&=(\sigma_{P_{H}} \ast (\delta_{H})_{P_{H}})(\varphi_{\pi_{H}})\\
&=\int_{G}{\int_{G}{\varphi_{\pi_{H}}(st) d\sigma_{P_{H}}(s) d(\delta_{H})_{P_{H}}}}\\
&=\int_{G}{\int_{G}{(L_{s^{-1}}\varphi_{\pi_{H}})(t) d(\delta_{H})_{P_{H}}(t) d\sigma_{P_{H}}(s)}}\\
&=\int_{G}{\int_{G/H}{P_{H}(L_{s^{-1}}\varphi_{\pi_{H}})(xH) d(\delta_{H})(xH) d\sigma_{P_{H}}(s)}}\\
&=\int_{G}{\int_{G/H}{\int_{H}{L_{s^{-1}}\varphi_{\pi_{H}}(xh)dh d(\delta_{H})(xH) d\sigma_{P_{H}}(s)}}}\\
&=\int_{G}{\int_{H}{L_{s^{-1}}\varphi_{\pi_{H}}(\eta h)dh d(\delta_{H})(xH) d\sigma_{P_{H}}(s)}},
\end{align*}
for some $ \eta \in H $. Therefore, because of $ dh $ is invariant, we have
\begin{align*}
\sigma \ast \delta_{H}(\varphi)
&=\int_{G}{\int_{H}{L_{s^{-1}}\varphi_{\pi_{H}}(\eta h)dh d\sigma_{P_{H}}(s)}}\\
&=\int_{G}{\int_{H}{\varphi_{\pi_{H}}(sh)dhd\sigma_{P_{H}}(s)}}\\
&=\int_{G}{\varphi_{\pi_{H}}(s)d\sigma_{P_{H}}(s)}\\
&=\int_{G/H}{\int_{H}{\varphi_{\pi_{H}}(xh)dh d\sigma(xH)}}\\
&=\int_{G/H}{\varphi(xH)d\sigma(xH)}\\
&=\sigma(\varphi).
\end{align*}
Thus, for all $ \sigma $ in $ M(G/H) $ we have $ \sigma \ast \delta_{H}=\sigma $. 
\end{proof}
\begin{corollary}\label{co17}
Let $H$ be a compact subgroup of a locally compact group $G$. If $ \sigma $ is a two-sided identity in the algebra $ M(G/H) $, then $ \sigma=\delta_{H} $.
\end{corollary}
\begin{proof}
Since $ \delta_{H}=\delta_{H} \ast \sigma=\sigma \ast \delta_{H}=\sigma $, the last equality is satisfied by considering the Lemma \ref{le11}.
\end{proof}
Generaly, $ \delta_{H} $ is not a left identity in the algebra $ M(G/H) $. Hence $ (M(G/H), *) $ failes to be a unital Banach algebra.
\begin{corollary}\label
The Banach algebra $ (M(G/H), *) $ is not an involutive algebra.
\end{corollary}
\begin{proof}
Since $ \sigma * \delta_{H}=\sigma $ for all $ \sigma \in M(G/H) $, it follows that $ \left(\sigma * (\delta_{H})^{*} \right) ^{*}=\delta_{H} * \sigma^{*} =\sigma^{*} $. Thus $ \delta_{H}*\sigma=\sigma $, for all $ \sigma \in M(G/H) $. This implies that $ \delta_{H} $ is a left identity, a contradiction.
\end{proof}
\begin{proposition}
Let $H$ be a compact subgroup of a locally compact group $ G $. 
\begin{enumerate}
\item[(i)] $ \delta_{xH} \ast \delta_{yH}=\delta_{xyH}$ is satisfied for all $ x, y \in G $ if and only if $ H $ is normal.
\item[(ii)] $ \delta_{H} $ is an identity for the Banach algebra $ (M(G/H), *) $ if and only if $ H $ is normal.
\end{enumerate}
\end{proposition}
\begin{proof}
(i) Let the equality $ \delta_{xH} \ast \delta_{yH}=\delta_{xyH}$ hold for all $ x, y \in G $. If $ H $ is not normal subgroup of $ G $, then by Urysohn's Lemma there exists $ \varphi $ in $ C_c(G/H) $ which $ \varphi(xH)=1 $ and $ \varphi(\eta xH)=0 $. Hence, by the hypothesis, $ \delta_{\eta H} \ast \delta_{xH}(\varphi)=\delta_{xH}(\varphi)=\varphi(xH)=1$. On the other hand $ \delta_{\eta H} \ast \delta_{xH}(\varphi)=\delta_{\eta xH}(\varphi)=\varphi(\eta xH)=0 $ and this is impossible. Thus $ H $ is normal.\\
For the converse, let $ H $ be a normal subgroup of $ G $. Considering the Lemma \ref{le11} and the Corollary \ref{co17}, it is enough to show that $ \delta_{H} $ is a left multiplicative identity in $ M(G/H) $. Let $ \varphi $ be in $ C_c(G/H) $ and $ \sigma $ be an arbitrary element in $ M(G/H) $. Then
\begin{align*}
(\delta_{H} \ast \sigma)(\varphi)
&=R_{H}((\delta_{H})_{P_{H}} \ast \sigma_{P_{H}})(\varphi)\\
&=((\delta_{H})_{P_{H}} \ast \sigma_{P_{H}})(\varphi_{\pi_{H}})\\
&=\int_{G}{\int_{G}{\varphi_{\pi_{H}}(st) d(\delta_{H})_{P_{H}}(s) d\sigma_{P_{H}}(t)}}\\
&=\int_{G}{\int_{G}{(R_{t}\varphi_{\pi_{H}})(s) d(\delta_{H})_{P_{H}}(s) d\sigma_{P_{H}}(t)}}\\
&=\int_{G}{\int_{G/H}{P_{H}(R_{t}\varphi_{\pi_{H}})(xH) d(\delta_{H})(xH) d\sigma_{P_{H}}(t)}}\\
&=\int_{G}{\int_{G/H}{\int_{H}{R_{t}\varphi_{\pi_{H}}(xh)dh d(\delta_{H})(xH) d\sigma_{P_{H}}(t)}}}\\
&=\int_{G}{\int_{H}{R_{t}\varphi_{\pi_{H}}(\eta h)dh d(\delta_{H})(xH) d\sigma_{P_{H}}(t)}},
\end{align*}
for some $ \eta \in H $. Since $ H $ is normal and $ dh $ is invariant, we have
\begin{align*}
(\delta_{H} \ast \sigma)(\varphi)
&=\int_{G}{\int_{H}{R_{t}\varphi_{\pi_{H}}(h)dhd\sigma_{P_{H}}(t)}}\\
&=\int_{G}{\int_{H}{\varphi_{\pi_{H}}(ht)dhd\sigma_{P_{H}}(t)}}\\
&=\int_{G}{\int_{H}{\varphi(htH)dhd\sigma_{P_{H}}(t)}}\\
&=\int_{G}{\int_{H}{\varphi(hHt)dhd\sigma_{P_{H}}(t)}}\\
&=\int_{G}{\int_{H}{\varphi(tH)dhd\sigma_{P_{H}}(t)}}\\
&=\int_{G}{\varphi_{\pi_{H}}(t)d\sigma_{P_{H}}(t)}\\
&=\int_{G}{\int_{H}{\varphi(tH)dhd\sigma_{P_{H}}(t)}}\\
&=\int_{G/H}{\int_{H}{\varphi_{\pi_{H}}(th)dh d\sigma(tH) d\sigma_{P_{H}}(t)}}\\
&=\int_{G/H}{\varphi(xH)d\sigma(xH)}\\
&=\sigma(\varphi).
\end{align*}
This implies that $ \delta_{H} \ast \sigma=\sigma $, that is $ \delta_{H} $ is an identity for $ M(G/H) $.\\ 
(ii) Assume that $ H $ is a normal subgroup  of $ G $, the proof to show that $ M(G/H) $ has an identity is the same as the proof of the converse part in $ (i). $\\ 
Conversely, suppose that $ \sigma $ is the two-sided identity in $ M(G/H) $ and assume that $ H $ is not normal. Then there exists some $ \eta \in H $ and $ x \in G $ such that $ \eta xH\neq xH $. Now take a $ \varphi $ in $ C_c(G/H) $ with $ \varphi(xH)=1 $ and $ \varphi(\eta xH)=0 $, this is possible by Urysohn's Lemma. By $ (i) $ above $ \delta_{\eta H} \ast \delta_{xH}(\varphi)=\delta_{xH}(\varphi)=\varphi(xH)=1,$ and on the other hand $ \delta_{\eta H} \ast \delta_{xH}(\varphi)=\delta_{\eta xH}(\varphi)=\varphi(\eta xH)=0 $, a contradiction.  
\end{proof}
\begin{proposition}\label{pr7}
Let $H$ be a compact subgroup of a locally compact group $G$ and also let $  \lambda$ be a strongly quasi-invariant measure on $ G/H $. Fix $ \varphi $ in $ L^1(G/H, \lambda) $. Then $ \psi \mapsto \int_{G/H}{\psi(xH)\varphi(xH) d\lambda(xH)} $, $ \psi \in C_{c}(G/H) $, defines a bounded measure on $ G/H $. Denoting this measure by $ \lambda_{\varphi} $, the mapping  $ \varphi \mapsto \lambda_{\varphi} $ is an isometric injection on $ L^1(G/H, \lambda) $ into $ M(G/H) $.
\end{proposition}
\begin{proof}
Let $ \varphi $ be a non zero element of $ C_{c}(G/H) $. For all $ xH $ in $ G/H $, set $ \varphi_{n}(xH):=(\vert \varphi(xH) \vert / \Vert \varphi \Vert _{\infty})^{1/n}sgn(\overline{\varphi(xH)}) $ for all $ n\geqslant1 $. Clearly, $ \varphi_{n} \varphi \geqslant 0 $, $ \Vert \varphi_{n} \Vert_{\infty}\leqslant 1 $, and also $ \varphi_{n} \varphi \uparrow \vert \varphi \vert $ as $ n $ goes to $ \infty $. Hence by using the monotone convergence Theorem we have
\[ \int_{G/H} {\vert \varphi \vert d\lambda(xH)}=\lim \int_{G/H} {\varphi_{n} \varphi d\lambda(xH)} \leqslant \int_{G/H} {\Vert \varphi_{n}\Vert_{\infty} \varphi d\lambda(xH)} \leqslant \int_{G/H} { \varphi d\lambda(xH)}=\Vert \lambda_{\varphi} \Vert.\]
The inverse is obvious. Therefore, $ \Vert \lambda_{\varphi} \Vert=\Vert \varphi \Vert_{1} $ for all $ \varphi $ in $ C_{c}(G/H).$
The general case then follows by approximation for all $ \varphi $ in $ L^{1}(G/H) .$
\end{proof}
\begin{remark}\label{re8}The asserted inclusion in Proposition \ref{pr7} is reduced to an injection of the algebra $ L^{1}(G) $ into $ M(G) $ in the case of $ H=\lbrace e \rbrace $.        
\end{remark}
The relation between some of the induced measures on $ G/H $ and $ G $, in Proposition \ref{pr7} and Remark \ref{re8} respectively, has been proved in the following Lemma. 
\begin{lemma}\label{le9}
Fix $ \mu $ as a left Haar measure on $ G $ and $ \lambda $ as a strongly quasi-invariant measure on $ G/H $ with associated rho-function $ \rho $, then for all $ \varphi \in C_{c}(G/H) $ we have $ R_{H}(\mu_{(\varphi_{\pi_{H}})})=\lambda_{\varphi} $, where $ d\lambda_{\varphi}(xH)=\varphi(xH)d\lambda(xH) $ and $ d\mu_{(\varphi_{\pi_{H}})}(x)=\varphi_{\pi_{H}}(x)d\mu(x) $.\end{lemma}
\begin{proof}
Let $ \psi \in C_{c}(G/H), $ then we have
\begin{align*}
R_{H}(\mu_{(\varphi_{\pi_{H}})})(\psi)
&=\mu_{(\varphi_{\pi_{H}})}(\psi_{\pi_{H}})\\
&=\int_{G}{\psi_{\pi_{H}}(x) d\mu_{(\varphi_{\pi_{H}})}(x)}\\
&=\int_{G}{\psi_{\pi_{H}}(x)\varphi_{\pi_{H}}(x)d\mu(x)}\\
&=\int_{G/H}{\int_{H}{\frac{\psi_{\pi_{H}}(xh)\varphi_{\pi_{H}}(xh)}{\rho(xh)}dh d\lambda(xH)}}\\
&=\int_{G/H}{\varphi(xH)\int_{H}{\frac{\psi_{\pi_{H}}(xh)}{\rho(xh)}}dh d\lambda(xH)}\\
&=\int_{G/H}{P_{H}(\psi_{\pi_{H}})(xH)d\lambda_{\varphi}(xH)}\\
&=\lambda_{\varphi}(P_{H}(\psi_{\pi_{H}}))\\
&=(\lambda_{\varphi})(\psi).
\end{align*}
\end{proof}
Consider the notations as in the Proposition \ref{pr7}. Put
\begin{align*}
\Lambda:=\lbrace \lambda_{\varphi} : \varphi \in L^{1}(G/H)     \rbrace.
\end{align*}
\begin{theorem}
Let $ \lambda $ be a strongly quasi-invariant measure on $ G/H $ arises from the rho-function $ \rho $. The Banach algebra $\left(  L^{1}(G/H), \lambda \right) $ is a two-sided ideal of the Banach algebra $ M(G/H) $.  
\end{theorem}
\begin{proof}
By the Proposition \ref{pr7} there is a one to one corresponding between $ L^{1}(G/H) $ and the range of the injection $ \varphi\mapsto\lambda_{\varphi} $, so it is enough to show that $ \Lambda $ is a two-sided ideal of $M(G/H) $.  To do this it is enough to show that $ \lambda_{\varphi} \ast \sigma <<\lambda $ for all $\lambda_{\varphi} \in \Lambda $ and $ \sigma \in M(G/H) $, since $ \Lambda $ consists precisely of those $ \sigma \in M(G/H) $ such that $ \sigma <<\lambda $. Now, by the definition of convolution on $ M(G/H),$ we have $ \lambda_{\varphi} \ast \sigma=R_{H}(\left( \lambda_{\varphi}\right) _{P_{H}}\ast \sigma_{P_{H}}) $, but
$ \left( \lambda_{\varphi}\right) _{P_{H}} \ast \sigma_{P_{H}}=\left( \lambda_{\varphi} \ast \sigma \right) _{P_{H}} \in M(G : H) $ and so by Proposition \ref{pr22} it equals to $ \mu_{\psi_{\pi_{H}}} $ for some $ \psi$ in  $L^{1}(G/H) $. Thus $ R_{H}(\left( \lambda_{\varphi}\right) _{P_{H}} \ast \sigma_{P_{H}})=R_{H}(\mu_{\psi_{\pi_{H}}})=\lambda_{\psi} \in \Lambda.$ Therefore, $ \lambda_{\varphi} \ast \sigma=\lambda_{\psi} << \lambda$ and the proof is complete.
\end{proof}
Fix a strongly quasi-invariant measure $ \lambda $ on $ G/H $ arises from the rho-function $ \rho $. Here and in the rest of sequel we set 
\begin{center}
$ C_c^{\rho}(G : H)=\lbrace \varphi_{\pi_{H}}^{\rho}:=\varphi \circ \pi_{H}\cdot \rho^{1/p} : \varphi \in C_c(G/H) \rbrace$,\\
\end{center}
and also take $ L^{p}(G : H)=\overline{C_c^{\rho}(G : H)}^{\Vert \cdot \Vert_{p}}, $ for all $1 \leq p < \infty. $
In a similar calculation in \cite{3, 7, 8}, one can see that $ C_c^{\rho}(G : H) $ is a left ideal of the algebra $ C_c(G) $. Then $ T_{H}^{p}: L^{p}(G) \rightarrow L^{p}(G/H) $ defined by $ T_{H}^{p}(f)(xH)=\int_{H}{\frac{f(xh)}{\rho(xh)^{1/p}}dh} $ is a surjective and bounded operator with $ \Vert T_{H}^{p} \Vert \leq 1 $. 
Consider the surjectivity of $ T_{H}^{p}: L^{p}(G : H) \rightarrow L^{p}(G/H) $, for all $ 1 \leq p < \infty $. By using the Proposition 3. 39 in \cite{5}, we know that $ \varphi_{\pi_{H}} \ast \psi_{\pi_{H}} $ belongs to $ L^{p}(G : H) $. Then one can define:
\begin{align*}
\varphi \ast \psi(xH)
&=T_{H}^{p}(\varphi_{\pi_{H}}\ast \psi_{\pi_{H}})(xH)=\int_{G/H}{\int_{H}{\varphi(yH)\psi(hy^{-1}xH)\left( \frac{\rho(hy^{-1}x)}{\rho(x)}\right) ^{1/p}dhd\lambda(yH)}},
\end{align*}
for all $ \varphi \in L^{1}(G/H) $ and $ \psi \in L^{p}(G/H) $.
Let $ \sigma \in M(G/H) $ and $ \varphi \in L^{p}(G/H) $, there exist $ \sigma_{P_{H}}\in M(G : H) $ and $ \varphi_{\pi_{H}}\in L^{p}(G : H) $ such that $ T_{H}^{p}(\varphi_{\pi_{H}}) $ and $ R_{H}(\sigma_{P_{H}})=\sigma $. Then we define the function $ \sigma \ast \varphi $ in a natural way as follows:
\[ \sigma \ast \varphi (xH)=T_{H}^{p}( \sigma_{P_{H}} \ast \varphi_{\pi_{H}})(xH).\]
But we have the following calculation:
\begin{align*}
\sigma \ast \varphi (xH)
&=T_{H}^{p}( \sigma_{P_{H}} \ast \varphi_{\pi_{H}})(xH)\\
&=\int_{H}{\frac{( \sigma_{P_{H}} \ast \varphi_{\pi_{H}})(x\eta)}{\rho(x\eta)^{1/p}}d\eta}\\
&=\int_{H}{\frac{\int_{G}{\varphi_{\pi_{H}}(y^{-1}x\eta)d\sigma_{P_{H}}(y)}}{\rho(x\eta)^{1/p}}d\eta}\\
&=\int_{H}{\frac{1}{\rho(x)^{1/p}}\int_{G/H}{\int_{H}{\varphi_{\pi_{H}}(h^{-1}y^{-1}x\eta)dh d\sigma(yH)d\eta}}}\\
&=\int_{G/H}{\int_{H}{\int_{H}{\varphi(h^{-1}y^{-1}xH)\left(  \frac{\rho(h^{-1}y^{-1}x)}{\rho(x)}\right)^{1/p}dh dh d\sigma(yH)}}}\\
&=\int_{G/H}{\int_{H}{\varphi(h y^{-1}xH)\left( \frac{\rho(h y^{-1}x)}{\rho(x)}\right) ^{1/p}dh d\sigma(yH)}},
\end{align*}
and in a similar calculation we get
\begin{align*}
\varphi \ast \sigma
=\int_{G/H}{\Delta(y^{-1})\int_{H}{\varphi(x h y^{-1}H)\left( \frac{\rho(x h y^{-1})}{\rho(x)}\right) ^{1/p}dh d\sigma(yH)}},
\end{align*}
\begin{proposition}
Suppose $ 1 \leq p < \infty $ and let $ \sigma \in M(G/H) $ and also $ \varphi \in L^{p}(G/H) $. Then $ \sigma \ast \varphi \in L^{p}(G/H)  $ and $ \Vert \sigma \ast \varphi \Vert_{p} \leq \Vert \sigma \Vert \Vert \varphi \Vert_{p} $
\end{proposition}
\begin{proof}
Fix $ \sigma \in M(G/H) $ and $ \varphi \in L^{p}(G/H) $. Considering definition of $ T_{H}^{p} $, the function $ \sigma \ast \varphi $ belongs to $ L^{p}(G/H) $. Using the fact that the mapping $ T_{H}^{p} $ and $ R_{H} $ are isometric on $ L^{p}(G : H) $ and $ M(G : H) $, respectively, we get
\[\Vert T_{H}^{p}( \sigma_{P_{H}} \ast \varphi_{\pi_{H}}) \Vert_{p}=\Vert \sigma_{P_{H}} \ast \varphi_{\pi_{H}} \Vert_{p}\leq\Vert \sigma_{P_{H}} \Vert\Vert \varphi_{\pi_{H}} \Vert_{p}=\Vert R_{H}\left( \sigma_{P_{H}}\right)  \Vert\Vert T_{H}^{p}\left( \varphi_{\pi_{H}}\right)\Vert_{p}. \]
This implies $\Vert \sigma \ast \varphi \Vert_{p}\leq\Vert \sigma \Vert\Vert \varphi \Vert_{p}.$
\end{proof}

\end{document}